\documentclass{amsart}

\usepackage{amssymb,color}
\usepackage{eepic,epic,times,epsfig,
fancybox,graphicx,pifont,pxfonts,txfonts,bbm}
\usepackage{amsfonts}
\usepackage{amsmath}
\usepackage{euscript}
\usepackage{enumerate}
\usepackage{graphics}

\newtheorem{theorem}{Theorem}[section]
\newtheorem{lemma}[theorem]{Lemma}

\newtheorem{prop}[theorem]{Proposition}

\newtheorem*{Theorem1'}{Theorem 1'}

\theoremstyle{definition}
\newtheorem{definition}[theorem]{Definition}

\theoremstyle{remark}

\numberwithin{equation}{section}

\newcommand \Hom{{\mathrm {Hom}}}

\newcommand \GL{{\mathrm {GL}}}

\newcommand \gl{{\mathfrak {gl}}}
\renewcommand \sl{{\mathfrak {sl}}}

\newcommand \g{{\mathfrak {g}}}

\newcommand \h{{\mathfrak {h}}}
\renewcommand \a{{\mathfrak {a}}}

\newcommand \la{{\lambda}}

\newcommand \al{{\alpha}}
\newcommand \be{{\beta}}
\newcommand \de{{\delta}}
\newcommand \ga{{\gamma}}

\newcommand \B{{\mathcal B}}

\newcommand \M{{\mathcal M}}
\newcommand \ad{{\mathrm {ad}}}

\renewcommand \k{{\mathfrak {k}}}
\renewcommand \h{{\mathfrak {h}}}

\begin{document}

\title[Indecomposable modules of solvable Lie algebras]{Indecomposable modules of solvable Lie algebras}

\author{Paolo Casati}
\address{Dipartimento di Matematica e Applicazioni, University of Milano-Bicocca, Milano, Italy}
\email{paolo.casati@unimib.it}

\author{Andrea Previtali}
\address{Dipartimento di Matematica e Applicazioni, University of Milano-Bicocca, Milano, Italy}
\email{andrea.previtali@unimib.it}

\author{Fernando Szechtman}
\address{Department of Mathematics and Statistics, Univeristy of Regina, Canada}
\email{fernando.szechtman@gmail.com}





\begin{abstract} We classify all uniserial modules of the solvable Lie algebra
$\g=\langle x\rangle \ltimes V$, where~$V$ is an abelian Lie algebra
over an algebraically closed field of characteristic 0 and~$x$ is an arbitrary automorphism of $V$.
\end{abstract}

\maketitle

\section{Introduction}

Let $F$ be an algebraically closed field of characteristic 0. All vector spaces, including all Lie algebras and their modules, are assumed to be finite dimensional over $F$.\par
Recall that a module is said to be indecomposable if
it cannot be decomposed as the direct sum of two non-trivial submodules.
Naturally, knowing all indecomposable modules of a given Lie algebra would provide a complete description of all its modules.
Unfortunately,  the problem of classifying all indecomposable modules of a given Lie algebra -that is not semisimple or one-dimensional- is virtually unsolvable, even in the case of the two-dimensional abelian Lie algebra, as observed in a celebrated paper by Gelfand and Ponomarev \cite{GP}.\par
In spite  of this fact, many types of indecomposable modules of non-semisimple Lie algebras have been
recently classified, see for example \cite{CGS,CMS,CS,CS1,D,DdG,DP,DR,P}\par
In all these papers the central idea is to consider particular classes of indecomposable modules for which a complete classification can be achieved.
Besides the irreducible modules, the simplest type of indecomposable module is,
in a certain sense, the uniserial one. This is a module having a unique composition series, i.e. a non-zero module whose submodules
form a chain. Alternatively, such modules can be  defined as follows.\par
Let $\g$ be a given Lie algebra
and let $U$ be a non-zero $\g$--module.  The  socle series $$0 = \mathrm{soc}_0(U) \subset \mathrm{soc}_1(U) \subset \cdots \subset  \mathrm{soc}_k(U) = U$$  of $U$  is  inductively
defined by declaring
$\mathrm{soc}_i(U)/\mathrm{soc}_{i-1}(U )$  to be the socle of $ U/\mathrm{soc}_{i-1}(U )$, that is, the sum of all irreducible submodules of
$U /\mathrm{soc}_{i-1}(U )$ , for $1\leq  i \leq k$. Then $U$ is uniserial if and only if the socle series of
$U$ has irreducible factors. \par
In the last years,  the classification of the uniserial  modules of important classes of solvable and perfect Lie algebras  has been achieved in various research papers \cite{CGS,CS,CS1,Pi,C}.
In particular, \cite{Pi} and \cite{C} classify a wider class of modules, called cyclic in  \cite{Pi}
and perfect cyclic in \cite{C}, over the perfect Lie algebras $\mathfrak{sl}(2)\ltimes F^2$ and
$\mathfrak{sl}(n+1)\ltimes F^{n+1}$, for $F=\mathbb{C}$, respectively.\par
The aim of this paper is to proceed further in the study of uniserial modules. We shall, indeed, classify the uniserial modules of a distinguished class
of solvable Lie algebras, namely those of the form $\g=\langle x\rangle \ltimes V$, where $V$ is an abelian Lie algebra and $x$ is an arbitrary automorphism of $V$.\par

A proper ideal $\a$ of $\g$ is of the form $\a=W$, where $W$ is an $x$-invariant subspace of~$V$. Thus, either $W=V$ and
$\g/\a\cong \langle x\rangle$ is one-dimensional, or else $\g/\a\cong \langle \overline{x}\rangle\ltimes  \overline{V}$, where $\overline{V}=V/W\neq (0)$
and $\overline{x}$ is the automorphism that $x$ induces on~$\overline{V}$.

We know from \cite{CS} all uniserial modules over an abelian Lie algebra as well as all uniserial $\g$-modules when
$x$ is diagonalizable. Thus, it suffices to classify all \emph{faithful} uniserial $\g$-modules when $x$ is
not diagonalizable. In this regard, our main results are as follows.

In \S\ref{sec1} we construct a family of non-isomorphic faithful uniserial representations of $\g$ when $x$ acts on $V$ via a single Jordan block of size $n>1$. This family consists of all matrix representations
\begin{equation}
\label{listrep}
R_{\al,k,X}\to\gl(n+1),\; R_{\al,n}\to\gl(n+1),\; R_{\al,1}\to\gl(n+1),
\end{equation}
where
$$
\al\in F,\; 1<k<n,\; X\in M_{k-1,n-k},
$$
as well as the matrix representations
\begin{equation}
\label{listrep2}
R_{\al,a}\to\gl(n+2),
\end{equation}
which exist only for odd $n$, and where
$$
\al\in F,\; a=(a_1,\dots,a_n),\; a_1=1,\; a_i=0\text{ for all even }i.
$$

In \S\ref{sec2}, we show that if $x$ acts on $V$ via a single Jordan block of size $n>1$, then every faithful uniserial $\g$-module is isomorphic to one and only one of the representations appearing in (\ref{listrep}) and (\ref{listrep2}).

In \S\ref{sec3} we deal with the general case.
By our results from \S\ref{sec2}, we may assume that $x$ has $e$ Jordan blocks, where $e>1$. Moreover, as indicated above,
we may also assume that $x$ is not diagonalizable. Under these assumptions. Theorem \ref{main2} gives necessary and sufficient
conditions for $\g$ to have a faithful uniserial module and classifies all such modules whenever these conditions are satisfied.

Indeed, let
\begin{equation}
\label{asjo}
V=V_1\oplus\cdots\oplus V_e
\end{equation}
be a decomposition of $V$ into indecomposable $F[x]$-submodules (this means that $x$ acts on each $V_i$ via a Jordan block) of dimensions $$n=n_1\geq\dots\geq n_e,$$
where $n>1$ because $x$ is not diagonalizable. For each $1\leq i\leq e$, consider the subalgebra $\g_i=\langle x_i\rangle \ltimes V_i$ of $\g$, where $x_i=x|_{V_i}$.

Suppose that $\g$ has a faithful uniserial representation $R:\g\to\gl(d)$. Then the restriction $R_1:\g_1\to\gl(d)$ is already uniserial.
In particular, $$d=n+1\text{ or }d=n+2.$$
In the first case, $R_1$ is isomorphic to a unique $R_{\al,k,X}$, the automorphism
$x$ has a single eigenvalue $\la$, and the Jordan decomposition of $x$ is
\begin{equation}
\label{asjo0}
J^{n}(\lambda)\oplus J^{n_2}(\lambda)\oplus\cdots\oplus J^{n_e}(\lambda),
\end{equation}
where
\begin{equation}
\label{asjo2}
n_2\leq n-2,\; n_3\leq n-4,\; n_4\leq n-6,\dots,\; n_e\leq n-2(e-1),
\end{equation}
and
\begin{equation}
\label{asjo3}
e\leq \mathrm{min}\{k,n+1-k\}.
\end{equation}
In the second case, $R_1$ is isomorphic to a unique $R_{\al,a}$, $n$ is odd, the automorphism $x$ has two eigenvalues $\la$ and $2\la$, and the Jordan decomposition of $x$ is
\begin{equation}
\label{asjo4}
J^{n}(\lambda)\oplus J^{1}(2\lambda),
\end{equation}
so that $e=2$ and $n_2=1$.

Conversely, if $x$ has a single eigenvalue, $1<k<n$, and (\ref{asjo2})-(\ref{asjo3}) are satisfied, then $R_{\al,k,X}$ can be extended
to a faithful uniserial representation of $\g$. In fact, let $\M_{k,n+1-k}$ be subspace of $\gl(n+1)$ consisting of all matrices
$$
\left(
           \begin{array}{cc}
             0 & N\\
             0 & 0\\
             \end{array}
         \right),\quad N\in M_{k\times n+1-k}.
$$
Let $v$ be a generator of the $F[x]$-module $V_1$ and set
$$A=R_{\al,k,X}(x)=\left(
           \begin{array}{cc}
             J^k(\al) & 0\\
             0 & J^{n+1-k}(\al-\la)\\
             \end{array}
         \right)\in\gl(n+1),\; E=R_{\al,k,X}(v)\in \M_{k,n+1-k}.
$$
Then the extensions of $R_{\al,k,X}$ to a faithful uniserial representation of $\g$ are given by all possible $F[t]$-monomorphisms $V\to\M_{k,n+1-k}$ such that $v\to E$, where $t$ acts via $\ad_\g x-\la 1_\g$ on $V$ and via $\ad_{\gl(n+1)}A -\la 1_{\gl(n+1)}$ on $\M_{k,n+1-k}$. Moreover,
all such extensions produce non-isomorphic representations of $\g$. Illustrative examples are provided in \S\ref{sec4}.

Likewise, if $x$ has Jordan decomposition (\ref{asjo4}), then $R_{\al,a}$ can be extended
to a faithful uniserial representation of $\g$. We determine all such extensions and prove that they produce non-isomorphic representations
of $\g$ (this case is much simpler than the above and no examples are required).

Finally, a necessary and sufficient condition for $\g$ to have a faithful uniserial representation
is that $x$ has Jordan decomposition (\ref{asjo0}) and (\ref{asjo2}) holds, or that $x$ has Jordan decomposition (\ref{asjo4})
and $n$ is odd.

Perhaps surprisingly, the representation theory of $\sl(2)$, and in particular the Clebsch-Gordan formula, plays a decisive role
in our study and classification of uniserial $\g$-modules.


\section{Construction of uniserial representations}\label{sec1}

Given $p\geq 1$ and $\al\in F$,
we write $J_p(\al)$ (resp. $J^p(\al)$) for the lower (resp. upper) triangular Jordan block of size $p$ and eigenvalue $\al$.
We also let $E^{i,j}\in\gl(p)$ stand for the matrix with entry $(i,j)$ equal to 1 and all other entries equal to~0.

We suppose throughout this section that $\g=\langle x\rangle \ltimes V$, where $V$ is an abelian Lie algebra and $x\in\gl(V)$ acts on $V$ via a single, lower triangular, Jordan block, say $J_n(\la)$, relative to a basis
$v_0,\dots,v_{n-1}$ of $V$. The case $\la=0$ is allowed. The multiplication table for $\g$ relative to its basis $x,v_{0},\dots,v_{n-1}$ is:
\begin{equation}
\label{rela}
[x,v_0]=\la v_0+v_1, [x,v_1]=\la v_1+v_2,\dots, [x,v_{n-1}]=\la v_{n-1}.
\end{equation}
We may translate (\ref{rela}) into
\begin{equation}
\label{rela1}
(\ad_\g\, x -\lambda 1_\g)^k v_0=v_k,\quad 0\leq k\leq n-1,
\end{equation}
and
\begin{equation}\label{cuc}
(\ad_\g x -\lambda 1_\g)^n v_0=0.
\end{equation}

\begin{prop}\label{sl2} Given positive integers $p,q$, let $\M_{p,q}$ be subspace of $\gl(p+q)$ consisting of all matrices
$$
\widehat{N}=\left(
           \begin{array}{cc}
             0 & N\\
             0 & 0\\
             \end{array}
         \right),\quad N\in M_{p\times q}.
$$
Given $\al,\la\in F$, set
$$
A=J^p(\al)\oplus J^{q}(\al-\la).
$$
Let $\theta$ stand for the endomorphism $\ad_{\gl(p+q)} A-\lambda 1_{\gl(p+q)}$ of $\gl(p+q)$ restricted to its invariant subspace  $\M_{p,q}$.
Then $\theta$ is nilpotent with elementary divisors
$$
t^{p+q-1},t^{(p+q-1)-2},t^{(p+q-1)-4},\dots,t^{(p+q-1)-2z},
$$
where
$$
z=\mathrm{min}\{p-1,q-1\}.
$$
Moreover, given any $N\in M_{p\times q}$,
the minimal polynomial of $\widehat{N}$ with respect to $\theta$ is $t^{p+q-1}$ if and only if $N_{p,1}\neq 0$.
\end{prop}

\begin{proof} For $a\geq 0$, let $V(a)$ stand for the irreducible $\sl(2)$-module with highest weight $a$. The Clebsch-Gordan formula states
that
\begin{equation}
\label{mfre}
V(a)\otimes V(b)\cong V(a+b)\oplus V(a+b-2)\oplus V(a+b-4)\oplus\cdots\oplus V(a+b-2r),
\end{equation}
where
$$
r=\mathrm{min}\{a,b\}.
$$
Let
$$
e=\left(
           \begin{array}{cc}
             0 & 1\\
             0 & 0\\
             \end{array}
         \right), h=\left(
           \begin{array}{cc}
             1 & 0\\
             0 & -1\\
             \end{array}
         \right), f=\left(
           \begin{array}{cc}
             0 & 0\\
             1 & 0\\
             \end{array}
         \right)
$$
stand for the canonical basis of $\sl(2)$. It is well-known that $e$ acts nilpotently on $V(a)$
with a single elementary divisor, namely $t^{a+1}$. It follows from the Clebsch-Gordan formula that $e$
acts nilpotently on $V(a)\otimes V(b)$ with elementary divisors
$$
t^{a+b+1},t^{(a+b+1)-2},t^{(a+b+1)-4},\dots,t^{(a+b+1)-2r}.
$$
For $a\geq 0$, let $R_a:\sl(2)\to\gl(a+1)$ be the matrix representation afforded by $V(a)$ given by
\begin{equation}
\label{sr1}
R_a(h)=\mathrm{diag}(a,a-2,\dots,-a+2,-a),
\end{equation}
\begin{equation}
\label{sr2}
R_a(e)=J^{a+1}(0),
\end{equation}
\begin{equation}
\label{sr3}
R_a(f)=\mathrm{diag}(0,a,2(a-1),3(a-2),\dots,3(a-2),2(a-1),a)J_{a+1}(0).
\end{equation}
Now
$$
V(a)\otimes V(b)\cong V(a)^*\otimes V(b)\cong\Hom(V(a),V(b)),
$$
where
\begin{equation}
\label{sr4}
(y\cdot \phi)(v)=y\cdot\phi(v)-\phi(y\cdot v),\quad y\in\sl(2),\phi\in \Hom(V(a),V(b)),v\in V.
\end{equation}
It follows that
\begin{equation}
\label{mfre2}
V(a)\otimes V(b)\cong \M_{a+1,b+1},
\end{equation}
where
\begin{equation}
\label{sesa}
y\cdot  \left(\begin{array}{cc}
             0 & N\\
             0 & 0\\
             \end{array}
         \right)= \left[\left(\begin{array}{cc}
             R_a(y) & 0\\
             0 & R_b(y)\\
             \end{array}
         \right),  \left(\begin{array}{cc}
             0 & N\\
             0 & 0\\
             \end{array}
         \right)\right], \quad y\in\sl(2).
\end{equation}
On the other hand, letting $B=J^p(0)\oplus J^{q}(0)$, we readily verify that
$$
(\ad_{\gl(p+q)} B)\widehat{N}=(\ad_{\gl(p+q)} A-\la 1_{\gl(p+q)})\widehat{N},\quad N\in M_{p\times q},
$$
which means that $\theta$ is the restriction of $\ad_{\gl(p+q)} B$ to
$\M_{p,q}$. Setting $a=p-1$ and $b=q-1$ and using (\ref{sr2}) as well as (\ref{sesa}), we deduce that $\theta$ is nothing but the action of $e$ on $\M_{a+1,b+1}$. The stated elementary divisors for $\theta$ now follow from those of the action of $e$ on $V(a)\otimes V(b)$.


Using (\ref{sr1}) and (\ref{sesa}) we find that, for $0\leq i\leq r$, the $h$-eigenspace of $\M_{p,q}$ with eigenvalue $-(a+b)+2i$, say $S(i)$, consists of all $\widehat{Q}$ such that the entries of $Q$ outside of its $i$th lower diagonal are equal to 0. Here the 0th lower diagonal consists of position $(p,1)$, the 1st lower diagonal of positions $(p,2),(p-1,1)$, the 2nd lower diagonal of positions $(p,3),(p-1,2),(p-2,1)$, and so on.

Each lowest weight vector of $\M_{p,q}$ generates an irreducible $\sl(2)$-submodule. In view of the multiplicity-free decomposition (\ref{mfre})
and the isomorphism (\ref{mfre2}), we see that for each $0\leq i\leq r$, there is one and only one $\widehat{0}\neq \widehat{E(i)}\in S(i)$, up to scaling, such that
\begin{equation}\label{eid}
f\cdot \widehat{E(i)}=0.
\end{equation}
Letting $W(i)$ be the $\sl(2)$-submodule generated by $\widehat{E(i)}$, we have
\begin{equation}\label{comre}
\M_{p,q}=W(0)\oplus\cdots\oplus W(r).
\end{equation}
Given an arbitrary $N\in M_{p\times q}$, let us write $\widehat{N}$ in terms of (\ref{comre}). We have
$$
\widehat{N}=w(0)+w(1)\cdots+w(r),\quad w(i)\in W(i),
$$
where,
$$
w(0)=\alpha_0 \widehat{E(0)}+\alpha_1 e\cdot \widehat{E(0)}+\cdots+\alpha_{a+b} e^{a+b}\cdot \widehat{E(0)},\quad \alpha_i\in F.
$$
From the first part of the Theorem, we know that, relative to the action of $e$, the minimal polynomial of $\widehat{E(0)}$ is $t^{a+b+1}$, while
$t^{a+b-1}$ annihilates all $w(i)$, $i>0$. It follows that the minimal polynomial of $\widehat{N}$ is $t^{a+b+1}$ if and only if $\al_0\neq 0$.
On the other hand, given that $\widehat{E(i)}\in S(i)$, it follows that every $E(i)$, $i>1$, has entry $(p,1)$ equal to 0, whereas entry $(p,1)$ of $E(0)$ is not 0. Moreover, using (\ref{sr2}) and (\ref{sesa}) we find that if $\widehat{P}=e\cdot \widehat{Q}$, then entry $(p,1)$ of $P$ is equal to 0 for any~$Q$. Thus, $N_{p,1}\neq 0$
if and only if $\al_0\neq 0$, as required.
\end{proof}

\begin{prop}\label{conex} Given $\al\in F$, positive integers $p,q$, and $N\in M_{p\times q}$ such that $N_{p,1}\neq 0$, consider the linear map $R=R_{\al,p,q, N}:\g\to\gl(p+q)$
given by
$$
x\mapsto A=
\left(
  \begin{array}{c|c}
    J^p(\al) & 0 \\\hline
    0 & J^q(\al-\la) \\
  \end{array}
\right)
$$
and
$$
v_k\mapsto (\ad_{\gl(p+q)} A-\la 1_{\gl(p+q)})^k\left(\begin{array}{c|c}
             0 & N\\\hline
             0 & 0\\
             \end{array}
         \right),\quad 0\leq k\leq n-1.
$$
Then $R$ is a representation of $\g$ if and only if $p+q-1\leq n$, in which case $R$ is uniserial. Moreover, $R$ is a faithful
representation if and only if $p+q-1=n$.
\end{prop}

\begin{proof} By construction, $R$ preserves the following relations of $\g$: $$[v_i,v_j]=0,\quad (\ad_\g\, x -\lambda 1_\g)^k v_0=v_k,\; 0\leq k\leq n-1.$$
On the other hand, due to Proposition \ref{sl2}, $(\ad_{\gl(p+q)} A-\la 1_{\gl(p+q)})^{n}\widehat{N}=0$ if and only if $p+q-1\leq n$, which
means that $R$ preserves the last defining relation of $\g$, namely $$(\ad_\g x -\lambda 1_\g)^n v_0=0,$$ if and only if $p+q-1\leq n$.
Thus, condition $p+q-1\leq n$ alone determines whether $R$ is a representation or not. Suppose that indeed $p+q-1\leq n$. It is obvious
that $R$ is uniserial. Moreover, $R(v_0),\dots,R(v_{n-1})$ are linearly independent if and only if the minimal polynomial
of $\widehat{N}$ with respect to $\ad_{\gl(p+q)} A-\la 1_{\gl(p+q)}$ has degree~$n$. By Proposition \ref{sl2}, this happens if and only if $p+q-1=n$.
\end{proof}

Given $\al\in F$, $1\leq k\leq n$,
 and $X\in M_{k-1,n-k}$, we set $p=k$, $q=n+1-k$ and
$$
N=\left(
    \begin{array}{c|c}
      0 & X \\\hline
      1 & 0 \\
    \end{array}
  \right)\in M_{k,n+1-k}.
$$
Then Proposition \ref{conex} ensures that
$$R_{\al,k,X}=R_{\al,p,q,N}:\g\to\gl(n+1)$$
is a faithful uniserial representation of $\g$. In the extreme cases $k=n$ and $k=1$ there is no~$X$, and $N$ is respectively equal to
$$
\left(
    \begin{array}{c}
      0 \\
      \vdots \\
      0 \\
      1 \\
    \end{array}
  \right)
\in M_{n\times 1}\text{ and }(1,0,\dots,0)\in M_{1\times n}.
$$
The corresponding representations will be respectively denoted by $R_{\al,n}$ and $R_{\al,1}$

Given any $\al\in F$ and $a=(a_1,\dots,a_n)\in F^n$ such that $a_1=1$,
we consider the linear map $R_{\al,a}:\g\to\gl(n+2)$ defined as follows:
$$
x\mapsto A=\left(
    \begin{array}{c|c|c}
      \al & 0 & 0 \\\hline
      0 & J^n(\al-\la) &  0\\\hline
      0 & 0 & \al-2\la \\
    \end{array}
  \right),
$$
$$
v_k\mapsto (\ad_{\gl(n+2)} A-\la 1_{\gl(n+2)})^k
\left(
    \begin{array}{c|c|c}
      0 & a & 0 \\\hline
      0 & 0 & e_n\\\hline
      0 & 0 & 0 \\
    \end{array}
  \right), \quad 0\leq k\leq n-1,
$$
where $\{e_1,\dots,e_n\}$ is the canonical basis of the column space $F^n$.

\begin{lemma}\label{nod} $R_{\al,a}$ is a representation of $\g$ if and only if $n$ is odd and $a_{i}=0$ for all even $i$, in which case $R_{\al,a}$ is uniserial.
\end{lemma}

\begin{proof} By definition, $R_{\al,a}$ preserves relations (\ref{rela1}) and (\ref{cuc}). We next determine
 when $R_{\al,a}$ preserves relations $[v_i,v_j]=0$. Letting $N=J^n(0)$, we have
$$
(\ad_{\gl(n+2)} A-\la 1_{\gl(n+2)})^k
\left(
    \begin{array}{c|c|c}
      0 & a & 0 \\\hline
      0 & 0 & e_n\\\hline
      0 & 0 & 0 \\
    \end{array}
  \right)=
  \left(
    \begin{array}{c|c|c}
      0 & (-1)^kaN^k & 0 \\\hline
      0 & 0 & N^ke_n\\\hline
      0 & 0 & 0 \\
    \end{array}
  \right), \quad 0\leq k\leq n-1.
$$
Thus $[R(v_k),R(v_j)]=0$ iff $(-1)^kaN^{k+j}e_n=(-1)^jaN^{k+j}e_n$ iff $a_{n-k-j}=0$ for $k+j$ odd.
Since $a_1=1$, the last condition is equivalent to $n$ odd and $a_{2s}=0$, for any $s$.
This proves the first assertion. As uniseriality is clear, the proof is complete.
\end{proof}

Finally, in the extreme case $n=1$, given any $\al\in F$ and $\ell\geq 2$ we have the faithful uniserial representation $T_{\al,\ell}:\g\to\gl(\ell)$
given by
$$
x\mapsto\mathrm{diag}(\al,\al-\la,\dots,\al-(\ell-1)\la),\; v_0\mapsto J^\ell(0).
$$

\begin{definition}\label{superdiagonalblocks} Given positive integers $\ell,d,d_1,\dots,d_\ell$ such that $d_1+\cdots+d_\ell=d$ and $\ell>1$,
and a matrix $A\in M_d$, we consider $A$ as partitioned into $\ell^2$ blocks $A(i,j)\in M_{d_i\times d_j}$.
We say that $A$ is block upper triangular if $A(i,j)=0$ for all $i>j$, and strictly block upper triangular if $A(i,j)=0$ for all $i\geq j$.
If $0\leq i\leq \ell-1$, by the $i$th block superdiagonal of $A$ we mean the $\ell-i$ blocks $A(1,1+i),A(2,2+i),\dots,A(\ell-i,\ell)$.
\end{definition}

\begin{lemma}\label{orbz2} Given positive integers $\ell,d_1,\dots,d_\ell$ ,with $\ell>1$, set $J_i=J^{d_i}(0)$, $1\leq i\leq \ell$, and
let $G$ be the subgroup of $\GL(d)$ consisting of all $X_1\oplus\cdots \oplus X_\ell$
such that $X_i\in U(F[J_i])$. Let $E\in M_d$ be strictly block upper triangular, with diagonal blocks of sizes
$d_1\times d_1,\dots,d_\ell\times d_\ell$. Let
$E_1\in M_{d_1\times d_2},\dots,E_{\ell-1}\in M_{d_{\ell-1}\times d_\ell}$ be the blocks in the first block superdiagonal of $E$, and suppose that the bottom left corner entry of each $E_i$ is
non-zero.

Then $E$ is $G$-conjugate to a matrix $H$ such that each of the blocks $H_1,\dots,H_{\ell-1}$ in the first block superdiagonal of $H$
has first column equal to the last canonical vector, and $H_{\ell-1}$ has last row equal to the first canonical vector.
Likewise, $E$ is also $G$-conjugate to a matrix $H$ such that each of the blocks $H_1,\dots,H_{\ell-1}$ in the first block superdiagonal of $H$ has last row equal to the first canonical vector, and the first column of $H_1$ is equal to the last canonical vector.
\end{lemma}

\begin{proof} Recall that $A\in F[J^p(0)]$ if and only if $A$ is upper triangular and all its superdiagonals have constant value. For instance, a typical element of $A\in F[J^4(0)]$ has the form
\begin{equation}
\label{formah}
A=\left(
  \begin{array}{cccc}
    \al & \be & \ga & \de \\
    0 & \al & \be & \ga \\
    0 & 0 & \al & \be \\
    0 & 0 & 0 & \al \\
  \end{array}
\right).
\end{equation}
It is clear that the unit group of $F[J^p(0)]$ acts transitively from the left (right) on the set of column (row) vectors of $F^p$ that have non-zero last (first) entry. Moreover, it is equally clear that if $U\in U(F[J^p(0)])$ and $B\in M_{q\times p}$ (resp. $B\in M_{p\times q}$) then the first column (resp. last row) of $BU$ (resp. $UB$) is that of $B$ scaled by a non-zero constant. It follows at once from these considerations that we can find
$X_1,\dots,X_{\ell-1}$ (resp. $X_2,\dots,X_{\ell}$) so that for any $X_\ell$ (resp. $X_1$) the resulting $X\in G$ will conjugate $E$ into a matrix $H$ such that the first column (resp. last row) of every $H_i$ is equal to a non-zero scalar multiple, say by $\al_i$, of the last (resp. first) canonical vector. Making a second selection of scalar matrices $Y_1,\dots,Y_\ell$ and conjugating $H$ by the resulting $Y=Y_1\oplus\cdots\oplus Y_\ell$, we can make all $\al_i=1$ above. Finally,
by suitably choosing $X_\ell$ (resp. $X_1$) with with 1's on the diagonal and taking all other $X_i=1_{d_i}$, we can make the last row (resp. first
column) of $H_{\ell-1}$ (resp. $H_1$) equal to the first (resp. last) canonical vector.
\end{proof}

\begin{prop}\label{nonas} Suppose $\la\neq 0$ and $n>1$. Then the representations $R_{\al,k,X}$, $R_{\al,n}$, $R_{\al,1}$ and $R_{\al,a}$ are non-isomorphic to each other.
\end{prop}

\begin{proof} Considering the eigenvalues of the image of $x$ as well as their multiplicities, the only possible isomorphisms are easily seen to be between $R_{\al,k,X}$ and $R_{\al,k,Y}$, or $R_{\al,a}$ and $R_{\al,b}$.

Suppose first $T\in\GL(n+1)$ satisfies
$$
T R_{\al,k,X}(y) T^{-1}=R_{\al,k,Y}(y),\quad y\in\g.
$$
Then $T$ commutes with $R_{\al,k,X}(x)=J^k(\al)\oplus J^{n+1-k}(\al-\la)$, and therefore $T=T_1\oplus T_2$,
where $T_1$ (resp. $T_2$) is a polynomial in $J^k(0)$ (resp. $J^{n+1-k}(0)$) with non-zero constant term. Thus
$$
T R_{\al,k,X}(v_0) T^{-1}=R_{\al,k,Y}(v_0)
$$
translates into
\begin{equation}
\label{ali}
T_1 \left(
    \begin{array}{c|c}
      0 & X \\\hline
      1 & 0 \\
    \end{array}
  \right)=\left(
    \begin{array}{c|c}
      0 & Y \\\hline
      1 & 0 \\
    \end{array}
  \right)T_2
\end{equation}
Explicitly writing $T_1$ and $T_2$, as in (\ref{formah}), we infer from (\ref{ali}) that $T_1=\al 1_{n+1}=T_2$, whence $X=Y$.

Suppose next $S\in\GL(n+2)$ satisfies
$$
S R_{\al,a}(y) S^{-1}=R_{\al,b}(y),\quad y\in\g.
$$
As above, $S=J^1(\be)\oplus S_2\oplus J^1(\ga)$, where $S_2$ is a polynomial in $J^n(0)$ with non-zero constant term and $\be,\ga$ are non-zero, but then
$$
S R_{\al,a}(v_0)=R_{\al,b}(v_0)S
$$
forces $S$ to be a non-zero scalar matrix, whence $a=b$.

\end{proof}

\section{Classification of uniserial representations}\label{sec2}

\begin{lemma}\label{rad} Let $T:\h\to\k$ be a homomorphism of Lie algebras. Then
$$
T((\ad_\h y -\mu 1_\h)^k z)=(\ad_\k T(y) -\mu 1_\k)^k T(z),\quad y,z\in\h,\mu\in F,k\geq 0.
$$
\end{lemma}

\begin{proof} This follows easily by induction.
\end{proof}

\begin{theorem}\label{main1} Consider the Lie algebra $\g=\langle x\rangle\ltimes V$, where $V$ is an abelian Lie algebra and $x\in\GL(V)$
acts on $V$ via a single Jordan block $J_n(\la)$, $\la\neq 0$. Let $R$ be a faithful uniserial representation of $\g$. Then

(a) If $n>1$ then $R$ is isomorphic to one and only one of the representations
$R_{\al,k,X}$, $R_{\al,n}$, $R_{\al,1}$, $R_{\al,a}$.

(b) If $n=1$ then $R$ is isomorphic to one and only one of the representations $T_{\al,\ell}$.
\end{theorem}

\begin{proof} Let $U$ be a faithful uniserial $\g$-module, say of dimension $d$. Lie's theorem ensures the existence of a basis $\B=\{u_1,\dots,u_d\}$ of $U$ such that the corresponding matrix representation $R:\g\to \gl(d)$ consists of upper triangular matrices.

Since $x\in\GL(V)$, we have $[\g,\g]=V$, whence $R(v)$ is strictly upper triangular for every $v\in V$. Set
$$
A=R(x)\text{ and }E_k=R(v_k),\; 0\leq k\leq n-1.
$$
In view of \cite[Lemma 2.2]{CS}, we may assume that $A$ satisfies:
\begin{equation}\label{aij}
A_{ij}=0\text{ whenever } A_{ii}\neq A_{jj}.
\end{equation}
Moreover, from \cite[Lemma 2.1]{CS}, we know that for every $1\leq i<d$ there is some $y_i\in\g$ such that
\begin{equation}\label{yi}
R(y_i)_{i,i+1}\neq 0.
\end{equation}

\bigskip

\noindent{\sc Step 1.} If $A_{i,i}\neq A_{i+1,i+1}$ then $(E_0)_{i,i+1}\neq 0$ and $A_{i,i}-A_{i+1,i+1}=\lambda$.

\bigskip

Indeed, (\ref{rela1}), (\ref{cuc}) and Lemma \ref{rad} imply
\begin{equation}\label{ew}
(\ad_{\gl(d)} A -\lambda 1_{\gl(d)})^k E_0=E_k,\quad 0\leq k\leq n-1,
\end{equation}
and
\begin{equation}\label{e0k2}
(\ad_{\gl(d)} A -\lambda 1_{\gl(d)})^n E_0=0.
\end{equation}
Since $A$ is upper triangular and $E_0$ is strictly upper triangular, (\ref{ew}) and (\ref{e0k2}) give
\begin{equation}\label{e0k3}
(A_{i,i}-A_{i+1,i+1}-\lambda)^k (E_0)_{i,i+1}=(E_k)_{i,i+1},\quad 0\leq k\leq n-1, 1\leq i<d.
\end{equation}
and
\begin{equation}\label{e0k4}
(A_{i,i}-A_{i+1,i+1}-\lambda)^n (E_0)_{i,i+1}=0,\quad 1\leq i<d.
\end{equation}
Fix $i$ such that $1\leq i<d$ and $A_{i,i}\neq A_{i+1,i+1}$. By (\ref{aij}), we have
\begin{equation}\label{cac}
A_{i,i+1}=0.
\end{equation}
Combining (\ref{yi}), (\ref{e0k3}) and (\ref{cac}) we obtain
\begin{equation}\label{cac2}
(E_0)_{i,i+1}\neq 0.
\end{equation}
From (\ref{e0k4}) and (\ref{cac2}) we deduce
$$
A_{i,i}-A_{i+1,i+1}=\lambda.
$$

\smallskip

\noindent{\sc Step 2.}  We have

\begin{equation}\label{desct}
A=A_1\oplus\cdots\oplus A_\ell,\quad A_i\in\gl(d_i),
\end{equation}
where each $A_i$ has scalar diagonal, say of scalar $\al_i$, and, setting $\al=\al_1$, we have
$$
\al_i=\al-(i-1)\la.
$$

\medskip

This follows at once from (\ref{aij}) and Step 1.

\bigskip

\noindent{\sc Step 3.} Let us write each $E_k$ in block form compatible with (\ref{desct}), that is, with diagonal blocks
of sizes $d_1\times d_1,\dots,d_\ell\times d_\ell$. Then all diagonal blocks of every $E_k$ are equal to 0.

\bigskip

Indeed, suppose $i\leq j$ and
$$
A_{i,i}=\cdots=A_{j,j}.
$$
Setting $U^r=\mathrm{span}\{u_1,\dots,u_{r}\}$, we see that the section
$$
U^{i,j}=U^j/U^{i-1}
$$
of $U$ is a (uniserial) $\g$-module of dimension $e=j-i+1$. Let $T:\g\to\gl(e)$ be the corresponding matrix representation
relative to the basis $u_i+U^{i-1},\dots,u_j+U^{i-1}$ of $U$. Then $T(x)$ is upper triangular with scalar diagonal,
so $\ad_{\gl(e)} T(x)$ is nilpotent. On the other hand, since $T$ is a Lie homomorphism, Lemma \ref{rad} gives
$$
(\ad_{\gl(e)} T(x)-\la 1_{\gl(e)})^n T(v_k)=0,\quad 0\leq k\leq n-1.
$$
It follows that every $T(v_k)$ is a generalized eigenvector of $\ad_{\gl(e)}T(x)$ for the distinct eigenvalues $\la$ and 0.
We infer that every $T(v_k)=0$. It follows that all diagonal blocks of every $E_k$ are equal to 0.

\bigskip

\noindent{\sc Step 4.} Referring to the block decomposition of $E_k$ used in Step 3,
if $i<j$ and $j\neq i+1$, then block $(i,j)$ of $E_k$ is 0 for all $0\leq k\leq n-1$.

\bigskip

Indeed, recalling Definition \ref{superdiagonalblocks}, we let
$$
S(1),S(2),\dots,S(\ell-1)
$$
be the subspaces of $\gl(d)$ corresponding to the block superdiagonals $1,2,\dots,\ell-1$, and
set
$$
S=S(1)\oplus\cdots\oplus S(\ell-1).
$$
Then $S(i)$ is the generalized eigenspace of $\ad_{\gl(d)}\, A$ acting on $S$ for the eigenvalue $i\la$, for all $1\leq i\leq \ell-1$.
On the other hand, every $E_k\in S$ by Step 3, while (\ref{ew}) and (\ref{e0k2})  imply that every $E_k$
belongs to the generalized eigenspace of $\ad_{\gl(d)}\, A$ for the eigenvalue~$\la$. We conclude that every $E_k$ is in $S(1)$.

\bigskip

\noindent{\sc Step 5.} We may assume without loss of generality that $A$ is in Jordan form
\begin{equation}
\label{jose}
J^{d_1}(\al)\oplus J^{d_2}(\al-\lambda)\oplus\cdots\oplus J^{d_\ell}(\al-(\ell-1)\lambda).
\end{equation}

\bigskip

Indeed, by (\ref{yi}) and Step 3, the first superdiagonal of every $A_i$ appearing in (\ref{desct}) consists entirely
of non-zero entries. Thus, for each $1\leq i\leq \ell$ there is $X_i\in\GL(d_i)$ such that
$$
X_i A_i X_i^{-1}=J^{d_i}(\al-(i-1)\la).
$$
Set
$$
X=X_1\oplus\cdots \oplus X_\ell\in\GL(d).
$$
Then $XAX^{-1}$ is equal to (\ref{jose}) and $XE_kX^{-1}$ is strictly block upper triangular with each block $(i,j)$, $j\neq i+1$, equal to 0.

\bigskip

\noindent{\sc Step 6.} $A$ has at least 2 Jordan blocks.

\bigskip

If not, $V$ is annihilated by $R$, by Steps 3 and 4, contradicting the fact that $R$ is faithful.

\bigskip

\noindent{\sc Step 7.} $d_i+d_{i+1}\leq n+1$ for all $1\leq i<\ell$.

\bigskip

Apply Proposition \ref{conex} to suitable sections of $U$.

\bigskip

\noindent{\sc Step 8.} Without loss of generality we may assume that the first column of each block along
the first block superdiagonal of $E_0$ is equal to the last canonical vector, and that the last row of the last of these
blocks is equal to the first canonical vector.

\bigskip

This follows from Lemma \ref{orbz2}.

\bigskip

\noindent{\sc Final Step when $n=1$. } Suppose $n=1$. Then all $d_i=1$ by Step 7, so
Steps 5, 6 and~8 yield that $R$ is isomorphic to a representation $T_{\al,\ell}$.
As these representations are clearly non-isomorphic to each other, the Theorem is proven in this case.

\bigskip

We assume for the remainder of the proof that $n>1$.

\bigskip

\noindent{\sc Step 9.} $A$ has at least one Jordan block of size $>1$.

\bigskip

If not, $d_i+d_{i+1}<n+1$ for all $i$ by Step 7. Since the $x$-invariant subspaces
of $V$ form a chain, it follows from Proposition \ref{conex} that $R(v_{n-1})=0$, contradicting the faithfulness of $R$.

\bigskip

\noindent{\sc Step 10.} Let $J^{d_i}(\al-(i-1)\la)$ be a Jordan block of $A$ of size $>1$ of $A$, as ensured
by Step~9. Then $i\geq \ell-1$ and $i\leq 2$.

\bigskip

Let us first see that $i\geq \ell-1$. Suppose, if possible, that $A$ has consecutive
Jordan blocks $J^a(\be),J^b(\be-\la),J^c(\be-2\la)$ with $a>1$.

\medskip

\noindent{Case 1.} $b=1$. Concentrating on a suitable section of $U$, as in the proof of Step 3, we see that $\g$ has a  matrix
representation $P:\g\to\gl(4)$ such that
$$
P(x)=\left(
         \begin{array}{cccc}
           \be & 1 & 0 & 0 \\
           0 & \be & 0 & 0 \\
           0 & 0 & \be-\la & 0 \\
           0 & 0 & 0 & \be-2\la \\
         \end{array}
       \right),\; P(v_0)=\left(
         \begin{array}{cccc}
           0 & 0 & 0 & 0 \\
           0 & 0 & 1 & 0 \\
           0 & 0 & 0 & 1 \\
           0 & 0 & 0 & 0 \\
         \end{array}
       \right).
$$
Here the shape of $P(E_0)$ is ensured by Steps 3, 4 and 8. Now
$$
P(v_1)=[P(x),P(v_0)]-\la P(v_0)=E^{1,3},
$$
which does not commute with $P(v_0)$, a contradiction.

\medskip

\noindent{Case 2.} $b>1$. Again, looking at a suitable section of $U$, we find a matrix representation
$Q:\g\to\gl(b+3)$
such that

\begin{equation}\label{comp1}
Q(x)=\left(
    \begin{array}{c|c|c}
      J^2(\be) & 0 & 0 \\\hline
      0 & J^b(\be-\la) &  0\\\hline
      0 & 0 & \be-2\la \\
    \end{array}
  \right),
\end{equation}
and
\begin{equation}\label{comp2}
Q(v_0)=\left(
    \begin{array}{c|c|c}
      0 & S & 0 \\\hline
      0 & 0 & u \\\hline
      0 & 0 & 0 \\
    \end{array}
  \right),
\end{equation}
where
\begin{equation}\label{comp3}
S=\left(
         \begin{array}{cccc}
           0 & * & \dots & * \\
           1 & * & \dots & * \\
         \end{array}
       \right)\in M_{2\times b},\; u=e_b\in F^b,
\end{equation}
and $\{e_1,\dots,e_b\}$ is the canonical basis of the column space $F^b$. By Lemma \ref{rad}, we have
\begin{equation}\label{comp6}
Q(v_k)=(\ad_{\gl(d)} Q(x)-\la 1_{\gl(d)})^k Q(v_0),\quad 0\le k\le b-1.
\end{equation}
Let $N=J^2(0)$, $M=J^b(0)$, $L$ left multiplication by $N$ and $R$ right multiplication by $M$.
Direct computation, using (\ref{comp1})-(\ref{comp3}), reveals that
\begin{equation}\label{comp4}
Q(v_k)=
\left(
    \begin{array}{c|c|c}
      0 & S_k & 0 \\\hline
      0 & 0 & e_{b-k} \\\hline
      0 & 0 & 0 \\
    \end{array}
  \right),
\end{equation}
where
\begin{equation}\label{comp5}
S_k=(L-R)^kS\in M_{2\times b}.
\end{equation}
Then $[Q(v_k),Q(v_j)]=0$ is equivalent to
\begin{equation}\label{comp5b}
S_ke_{b-j}=S_je_{b-k}.
\end{equation}
Since $L^2=0$, $S_k=(-1)^kSM^k+(-1)^{k-1}kNSM^{k-1}$, so
$$
S_k=\left(
         \begin{array}{cccccc}
           0 \dots 0 & (-1)^{k-1}k &      * & * & \dots & * \\
           0 \dots 0 &           0 & (-1)^k & * & \dots & * \\
         \end{array}
       \right),
$$
where the first non-zero column occurs in position $k$.
Taking first $(k,b)=(b-1,0)$ and then $(k,b)=(b-1,1)$ in (\ref{comp5b}), we respectively get
\begin{equation}\label{comp5c}
(-1)^{b-1}=1,\quad (-1)^b(b-1)=1,
\end{equation}
which is impossible. This proves that $i\geq \ell-1$.

The proof that $i\leq 2$ is entirely analogous. Alternatively, it can be obtained from above by duality. Indeed, the dual module $U^*$ is also
faithful and uniserial. The corresponding matrix representation, say $K:\g\to\gl(d)$, relative to the dual basis $\{u_1^*,\dots,u_d^*\}$, is given by
$$
K(y)=-R(y)',\quad y\in\g,
$$
the opposite transpose of $R(y)$. Conjugating each $K(y)$ by the block permutation matrix corresponding to the permutation
\begin{equation}
\label{ky}
1\leftrightarrow\ell, 2\leftrightarrow (\ell-1),\dots
\end{equation}
and further conjugating the resulting representation by a suitable block diagonal matrix,
we obtain a matrix representation $L:\g\to\gl(d)$, where the sizes of the Jordan blocks of $L(x)$ are those of $A$ in reversed order
according to (\ref{ky}). Thus, $i\leq 2$ follows from $i\geq \ell-1$.

\bigskip

\noindent{\sc Final Step when $n>1$.} $R$ is isomorphic to one and only one of the representations
$R_{\al,k,X}$, $R_{\al,n}$, $R_{\al,1}$ and $R_{\al,a}$.

\bigskip

Indeed, it follows at once from Step 10 that $\ell=2$, or $\ell=3$ and $d_1=d_3=1$.

Suppose first $\ell=2$. Then Proposition \ref{conex} ensures $d_1+d_2=n+1$. It now follows easily from Steps 3, 5 and 8 that $R$ is isomorphic to
$R_{\al,n}$ if $d_1=n$, to $R_{\al,1}$ if $d_1=1$, and to $R_{\al,d_1,X}$ if $1<d_1<n$, where $X$ is obtained by eliminating the first column
and last row of $E_0$,

Suppose next $\ell=3$ and set $e=d_2$. We have
 $$
A=\left(
    \begin{array}{c|c|c}
      \al & 0 & 0 \\\hline
      0 & J^{e}(\al-\la) & 0 \\\hline
      0 & 0 & \al-2\la \\
    \end{array}
  \right),\; E_0=\left(
    \begin{array}{c|c|c}
      0 & a & 0 \\\hline
      0 & 0 & b \\\hline
      0 & 0 & 0 \\
    \end{array}
  \right),
$$
where
$$a=(1,a_2,\dots,a_{e})\in F^{e},\;
b= \left(\begin{array}{c}
0 \\
\vdots \\
0 \\
1 \\
\end{array}
\right)\in F^{e}.
$$
By Lemma \ref{nod}, we see that $e$ is odd and $a_i=0$ for all even $i$. As $(\ad_{\gl(d)} A-\la 1_{\gl(d)})^{n-1} E_0\neq 0$ and $(\ad_{\gl(d)} A-\la 1_{\gl(d)})^{n} E_0=0$, we must have $e=n$, whence $R$ is isomorphic to $R_{\al,a}$.

Whether $\ell=2$ or $\ell=3$, uniqueness follows from Proposition \ref{nonas}.
\end{proof}

\section{The general case}\label{sec3}

Throughout this section $\g=\langle x\rangle\ltimes V$, where $x\in\GL(V)$. We wish to classify all uniserial $\g$-modules. As explained in the Introduction, we may restrict to analysing faithful modules in the case when $x$ is not diagonalizable.
Let
$$
V=V_1\oplus\cdots\oplus V_e,
$$
be a decomposition of $V$ into non-zero indecomposable $F[x]$-modules of dimensions
$$
n=n_1\geq n_2\geq \cdots\geq n_e\geq 1.
$$
Since $x$ is not diagonalizable, $n>1$. By Theorem \ref{main1}, we may restrict to the case $e>1$. For each
$1\leq i\leq e$, consider the subalgebra
$\g_i=\langle x_i\rangle \ltimes V_i$ of $\g$, where $x_i=x|_{V_i}$.

In what follows, a generator of any $V_i$ as $F[x]$-module will be simply referred to as a generator.

Given positive integers $p,q$, we write $\M_{p,q}$ for subspace of $\gl(p+q)$ consisting of all matrices
$$
\widehat{N}=\left(
           \begin{array}{cc}
             0 & N\\
             0 & 0\\
             \end{array}
         \right),\quad N\in M_{p\times q}.
$$
\begin{theorem}\label{main2} Assume that $x$ has $e$ Jordan blocks, where $e>1$, and is not diagonalizable.

\medskip

(1) Suppose that $\g$ has a faithful uniserial representation $R:\g\to\gl(d)$. Then the restriction $R_1:\g_1\to\gl(d)$ is already uniserial.
In particular, $$d=n+1\text{ or }d=n+2.$$
In the first case, $R_1$ is isomorphic to a unique $R_{\al,k,X}$, the automorphism
$x$ has a single eigenvalue $\la$, and $x$ has Jordan decomposition
\begin{equation}
\label{unosolo}
J_{n}(\lambda)\oplus J_{n_2}(\lambda)\oplus\cdots\oplus J_{n_e}(\lambda),
\end{equation}
where
\begin{equation}
\label{ere}
n_2\leq n-2,\; n_3\leq n-4,\; n_4\leq n-6,\dots,\; n_e\leq n-2(e-1),
\end{equation}
and
\begin{equation}
\label{ere2}
e\leq \mathrm{min}\{k,n+1-k\}.
\end{equation}
In the second case, $R_1$ is isomorphic to a unique $R_{\al,a}$, $n$ is odd, the automorphism $x$ has two eigenvalues $\la$ and $2\la$, and $x$ has Jordan decomposition
\begin{equation}
\label{dossolos}
J_{n}(\lambda)\oplus J_{1}(2\lambda),
\end{equation}
so that $e=2$ and $n_2=1$.

(2) Suppose, conversely, that $x$ has a single eigenvalue, $1<k<n$, and (\ref{ere})-(\ref{ere2}) are satisfied. Then $R_{\al,k,X}$ can be extended
to a faithful uniserial representation of $\g$. Let $v_0,\dots,v_{n-1}$ be a basis of $V_1$ relative to which the matrix of $x_1$ is $J_n(\la)$,
and set
$$A=R_{\al,k,X}(x)=J^k(\al)\oplus J^{n+1-k}(\al-\la)\in\gl(n+1),\;
E_0=R_{\al,k,X}(v_0)\in M_{k\times n+1-k}.$$
Then the extensions of $R_{\al,k,X}$
to a faithful uniserial representation $\g\to\gl(n+1)$ are given by all possible $F[t]$-monomorphisms $V\to\M_{k,n+1-k}$ such that $v_0\to E_0$, where $t$ acts via $\ad_\g x-\la 1_\g$ on $V$ and via $\ad_{\gl(n+1)}A -\la 1_{\gl(n+1)}$ on $\M_{k,n+1-k}$. Abstractly, these extensions are
given by all possible $F[t]$-monomorphisms
$$
F[t]/(t^n)\oplus F[t]/(t^{n_2})\oplus\cdots\oplus F[t]/(t^{n_e})\to
F[t]/(t^n)\oplus F[t]/(t^{n-2})\oplus\cdots\oplus F[t]/(t^{n-2(s-1)})
$$
which are the identity map on the first summand, where $s=\mathrm{min}\{k,n+1-k\}$.

Moreover, if $R:\g\to\gl(n+1)$ and $S:\g\to\gl(n+1)$ are faithful uniserial representations extending $R_{\al,k,X}$, then $R$ and $S$
are isomorphic if and only if they are equal.

\medskip

(3) Likewise, if $x$ has eigenvalues $\la$ and $2\la$, with $e=2$ and $n_2=1$, then $R_{\al,a}$ can be extended
to a faithful uniserial representation of $\g$. Let $0\neq w_0\in V_2$. Then the extensions of $R_{\al,a}$
to a faithful uniserial representation $\g\to\gl(n+2)$ are given by all possible functions
$$
w_0\to \be E^{1,n+2},\quad\be\in F,\; \be\neq 0.
$$
Moreover, if $R:\g\to\gl(n+2)$ and $S:\g\to\gl(n+2)$ are faithful uniserial representations extending $R_{\al,a}$, then $R$ and $S$
are isomorphic if and only if they are equal.

\medskip

(4) A necessary and sufficient condition for $\g$ to have a faithful uniserial representation
is that $x$ has Jordan decomposition (\ref{unosolo}) and (\ref{ere}) holds, or that $x$ has Jordan decomposition (\ref{dossolos})
and $n$ is odd.

\end{theorem}

\begin{proof} (1) Let $T:\g\to\gl(U)$ be a faithful uniserial representation, say of dimension~$d$.
Perusing the proof of Theorem \ref{main1}, we see that there is a basis $\B$ of $U$ such that the corresponding matrix
representation $R:\g\to\gl(d)$ satisfies:

$\bullet$ There are positive integers $\ell>1$ and $d_1,\dots,d_\ell$ such that
$$
d_1+\cdots+d_\ell=d,
$$
and relative to this decomposition each $R(y)$, $y\in\g$, is block upper triangular and each $R(v)$, $v\in V$, is strictly block upper triangular.

\medskip

$\bullet$ The diagonal blocks of $A=R(x)$ are $J^{d_1}(\al_1),\dots,J^{d_\ell}(\al_\ell)$, where each $\al_i-\al_{i+1}$ is an eigenvalue
of $x$ acting on $V$. Moreover, if $i<j$ and $\al_i\neq \al_j$ then block $(i,j)$ of $A$ is 0.

\medskip

$\bullet$ For each $1\leq i<\ell$ there is a generator $v$ such that block $(i,i+1)$ of $R(v)$ has non-zero bottom left entry.
Moreover, if $v$ is associated to the eigenvalue $\la$ of $x$, then necessarily $\al_i-\al_{i+1}=\la$.

\medskip

We claim that $d_i>1$ for at least one $i$. Suppose not. Then $A_{j,j+1}=0$ for all $1\leq j<d$. Since $R$ is uniserial and $V$
is abelian, there is a single generator $v$ such that $R(v)_{j,j+1}\neq 0$ for all $1\leq j<d$, whence
the diagonal entries of $A$ are in arithmetic progression of step $\la$, the eigenvalue associated to $v$. In particular, $A$ is diagonalizable.
Thus $\ad_{\gl(d)} A$ is diagonalizable, whence $\ad_{R(\g)} A$ is diagonalizable.
But $R$ is faithful, so $\ad_{\g} x$ is diagonalizable, which means that $x$ acts diagonalizably on $V$, against
the stated hypotheses.

By above, there is a diagonal block $J^a(\al)$ of $A$ such that $a>1$. Suppose, if possible, that $A$ has consecutive diagonal blocks
$J^a(\al),J^b(\be), J^c(\ga)$. Then $\la=\al-\be$ and $\mu=\be-\ga$ are eigenvalues of $x$.
Considering a suitable section of $U$ we obtain a representation $P:\g\to\gl(a+b+1)$ such that
$$
P(x)=\left(
    \begin{array}{c|c|c}
      J^{a}(\al) & 0 & z \\\hline
      0 & J^{b}(\al-\la) & 0 \\\hline
      0 & 0 & \al-(\la+\mu) \\
    \end{array}
  \right).
$$
Note that $z=0$ if $\la+\mu\neq 0$.

Let $S$ be the subspace of $\gl(a+b+1)$ consisting of all
strictly block upper triangular matrices. Let $S(1,2)$ be the subspace of all matrices in $S$ whose blocks $(i,j)\neq (1,2)$
are equal to 0, and define $S(2,3)$ and $S(1,3)$ likewise. Then the generalized eigenspace for the action of $\ad_{\gl(a+b+1)} P(x)$
on $S$ of eigenvalue $\la+\mu$ is $S(1,3)$. Moreover, the generalized eigenspace for the action of $\ad_{\gl(a+b+1)} P(x)$
on $S$ of eigenvalue $\la$ (resp. $\mu$) is $S(1,2)$ (resp. $S(2,3)$) if $\la\neq\mu$, and $S(1,2)\oplus S(2,3)$
if $\la=\mu$.

If a single generator $v$ has the property that both blocks $(1,2)$ and $(2,3)$ of $P(v)$ have non-zero bottom left entry, then, in particular, $\la=\mu$, and we may argue as in Step 10 of the proof of Theorem \ref{main1}
to see that the commutativity of $V$ is contradicted. Thus there exist generators $v$ and $w$, associated to eigenvalues $\la$ and~$\mu$, respectively, such that
$$
P(v)=\left(
    \begin{array}{c|c|c}
      0 & S & 0 \\\hline
      0 & 0 & u \\\hline
      0 & 0 & 0 \\
    \end{array}
  \right), P(w)=\left(
    \begin{array}{c|c|c}
      0 & T & 0 \\\hline
      0 & 0 & y \\\hline
      0 & 0 & 0 \\
    \end{array}
  \right),
$$
where
$$
S_{a,1}\neq 0, T_{a,1}=0,  y_b\neq 0.
$$
It should be noted that blocks $(1,3)$ of $P(v)$  and $P(w)$ are indeed 0, as $\la+\mu$ is different from $\la$ and~$\mu$, and
$P(v)$ and $P(w)$ are generalized eigenvectors for the action of $\ad_{\gl(a+b+1)} P(x)$ on~$S$ with eigenvalues
$\la$ and $\mu$, respectively.

Let $f$ be the dimension of the $F[x]$-submodule, say $W$, generated by $w$. It follows from Proposition \ref{conex} that $b\leq f$. Thus, the vectors $w_k=(\ad_\g x-\mu 1_\g)^k w$, $0\leq k\leq f-1$, form a basis of $W$ and $w_{b-1}$ is amongst them. We have
$$
P(w_{b-1})=(\ad_{\gl(a+b+1)} P(x)-\la 1_{\gl(a+b+1)})^{b-1}P(w_0).
$$
Since $T_{a,1}=0$ and $y_b\neq 0$, direct computation reveals that
$$
P(w_{b-1})=\left(
    \begin{array}{c|c|c}
      0 & Y & 0 \\\hline
      0 & 0 & h \\\hline
      0 & 0 & 0 \\
    \end{array}
  \right),
$$
where the last row of $Y$ is 0 and $h$ is a non-zero scalar multiple of first canonical vector of~$F^b$.
Thus block $(1,3)$ of $[P(v),P(w_{b-1})]$ is equal to
$$
Sh-Yu,
$$
whose last entry is non-zero since $S_{a,1}\neq 0$. This contradiction proves that $A$ cannot have consecutive diagonal blocks
$J^a(\al),J^b(\be), J^c(\ga)$ with $a>1$.
Arguing by duality, as in the proof of Step 10 of Theorem \ref{main1}, it follows that $A$ cannot have consecutive diagonal blocks $J^c(\ga),J^b(\be), J^a(\al)$ with $a>1$. We deduce that $\ell=2$, or $\ell=3$ and $d_1=1=d_3$.

Suppose first that $\ell=2$. The only eigenvalue of $\ad_{\gl(d)}A$ acting on $\M_{d_1,d_2}$ is $\la$.
 Since $R$ is faithful, the only eigenvalue of $x$ acting on $V$ must be $\la$. The case $d_1=1$ (resp. $d_2=1$), is easily seen to contradict the faithfulness of $R$, as $e>1$. Thus, $d_1>1$ and $d_2>1$. Let $v$ be a generator such that the bottom left entry of block (1,2) of $R(v)$ is non-zero.
It follows from Proposition \ref{conex} that $d_1+d_2=n+1$, where $n$ is the dimension of the $F[x]$-submodule of $V$ generated by $v$. Lemma \ref{rad} implies that the restriction of $R$ to $V$ is a monomorphism of $F[t]$-modules $V\to \M_{d_1,d_2}$,
where $t$ acts via $\ad_\g x-\la 1_\g$ on $V$ and via $\ad_{\gl(n+1)}A -\la 1_{\gl(n+1)}$ on $\M_{d_1,d_2}$.
It now follows from Proposition \ref{sl2} and the theory of finitely generated modules over a principal ideal domain,
that (\ref{ere2}) holds for $k=d_1$ and that the elementary divisors $t^{n_1},\dots,t^{n_e}$ of $V$ satisfy (\ref{ere}). It is clear that the restriction of $R$ to $\g_1$ is uniserial. By Theorem \ref{main1}, this restriction is isomorphic to a unique $R_{\al,k,X}$.

Suppose next that $\ell=3$ and $d_1=1=d_3$. Assume, if possible, that there exist generators $u,w$ such that
$$
R(u)=\left(
    \begin{array}{c|c|c}
      0 & u_1 & 0 \\\hline
      0 & 0 & u_2 \\\hline
      0 & 0 & 0 \\
    \end{array}
  \right), R(w)=\left(
    \begin{array}{c|c|c}
      0 & w_1 & 0 \\\hline
      0 & 0 & w_2 \\\hline
      0 & 0 & 0 \\
    \end{array}
  \right),
$$
where the first (reps. last) entry of $u_1$ (resp. $w_2$) is non-zero, and the last (reps. first) entry of $u_2$ (resp. $w_1$) is 0.
It follows that
$$
R([u,(\ad_\g x-\la 1_g)^{d_2-1}w])=[R(u),(\ad_{\gl(d)} R(x)-\la 1_{\gl(d)})^{d_2-1}R(w)]\neq 0,
$$
contradicting the fact that $V$ is abelian.

Since $R$ is uniserial, we deduce from above that there is a single generator~$u$
such that the bottom left entries of blocks $(1,2)$ and $(2,3)$ of $R(u)$ are non-zero. Let $\la$ be the eigenvalue associated to $u$
and set $c=d_2$. Since some $d_i>1$, it follows that $c>1$. We have
$$
A=\left(
    \begin{array}{c|c|c}
      \al & 0 & 0 \\\hline
      0 & J^c(\al-\la) & 0 \\\hline
      0 & 0 & \al-2\la \\
    \end{array}
  \right).
$$
Moreover, let $W$ be the $F[x]$-submodule of $V$ generated by $u$. Then
\begin{equation}
\label{spw}
R(w)=\left(
    \begin{array}{c|c|c}
      0 & w_1 & 0 \\\hline
      0 & 0 & w_2 \\\hline
      0 & 0 & 0 \\
    \end{array}
  \right),\quad w\in W.
\end{equation}
In the special case $w=u$, conjugating by a suitable block diagonal matrix, we may assume
$$u_1=(1,a_2,\dots,a_c)\text{ and }
u_2= \left(\begin{array}{c}
0 \\
\vdots \\
0 \\
1 \\
\end{array}
\right).
$$
Consider the subalgebra $\h=\langle y\rangle\ltimes W$ of $\g$, where $y=x|_W$. It is clear that the restriction, $S:\h\to\gl(d)$.
is uniserial. Since $S$ is a faithful representation, we see that $c=\dim(W)$. We see from Lemma \ref{nod} that $c$ is odd and $a_i=0$ for all even $i$, so $S$ is isomorphic to $R_{\al,a}$.

The Jordan decomposition of $\ad_{\gl(d)} A$ acting on the space of strictly block upper triangular matrices is
$$
J^c(\la)\oplus  J^c(\la)\oplus J^1(2\la).
$$
Suppose, if possible, that $x$ has another Jordan block with eigenvalue $\la$. Then there is a summand $V_i$ of $V$ different from $W$. Let $w\in V_i$ be an eigenvector of $\ad_\g x$. Since $[R(u),R(w)]=0$, we see that $w_1$ and $w_2$
are equal to the last and first canonical vectors, respectively, multiplied by the \emph{same} scalar. Since $c$ is odd, it follows that
$R(w)$ is a scalar multiple of $R((\ad_\g x-1_\g)^{c-1}u)$, which contradicts the faithfulness of~$R$.

Since $e>1$ and $R$ is faithful, we infer that the Jordan decomposition $\ad_g x$ acting on $V$ must be
\begin{equation}
\label{jox}
J^n(\la)\oplus J^1(2\la).
\end{equation}
By definition of $\g$, (\ref{jox}) is also the Jordan decomposition of $x$, so $c=n$, $e=2$ and $n_2=1$.

\medskip

(2) Thanks to the stated hypothesis we may use Proposition \ref{sl2} to produce an $F[t]$-monomorphism $V\to \M_{k,n+1-k}$ satisfying $v_0\mapsto R_{\al,k,X}(v_0)$, where $R_{\al,k,X}$ is understood to be a representation of $\g_1$. Any such monomorphism extends $R_{\al,e,X}$ to a faithful uniserial representation of $\g$. Conversely, Lemma \ref{rad} ensures that any extension of $R_{\al,e,X}$ to a faithful uniserial representation of $\g$ restricts to an $F[t]$-monomorphism $V\to \M_{k,n+1-k}$, which obviously satisfies $v_0\mapsto R_{\al,k,X}(v_0)$.

As shown in the proof of Proposition \ref{nonas}, the centralizer of $R(x),R(v_0)$ in $\gl(n+1)$ consists of scalar matrices, so different extensions of $R_{\al,k,X}$ from $\g_1$ to $\g$ yield non-isomorphic representations of $\g$.

\medskip

(3) The argument here is similar but much easier than the above.

\medskip

(4) The necessity of the stated conditions follows from item (1). Suppose that $x$ has Jordan decomposition (\ref{dossolos}) and $n$ is odd. Then $\g$ has a faithful uniserial representation by item (3). Suppose finally that $x$ has Jordan decomposition (\ref{unosolo}) and that (\ref{ere}) holds. We infer from (\ref{ere}) that $2e\leq n+1$, so $e<n$ and $e\leq\min\{e,n+1-e\}$. By hypothesis, we have $e>1$. It follows
from item (2) that $\g$ has a faithful uniserial representation.

\end{proof}

\section{Examples}\label{sec4}

Illustrative examples of Theorem \ref{main2} can be obtained by making explicit use of Proposition \ref{sl2}. Indeed, let us find $E(i)$
satisfying (\ref{eid}) for all $0\leq i\leq r=\min\{a,b\}$. As $E(i)\in S(i)$, there are scalars $\al_0,\dots,\al_i\in F$ such that
$$
E(i)=\underset{a+1-i\leq k\leq a+1}\sum \al_{a+1-k} E^{k,k+i-a}
$$
or, alternatively,
$$
E(0)=\al_0 E^{a+1,1}, E(1)=\al_1 E^{a,1}+\al_0 E^{a+1,2},
E(2)=\al_2 E^{a-1,1}+\al_1 E^{a,2}+\al_0 E^{a+1,3},
$$
$$
\dots, E(r)=\al_r E^{a+1-r,1}+\al_{r-1} E^{a+2-r,2}+\cdots+\al_0 E^{a+1,r+1}.
$$
If $i=0$ there is nothing to do. Suppose $i>0$. By means of (\ref{sr3}) and (\ref{sesa}), we find (\ref{eid}) equivalent to
$$
t(a+1-t)\al_t=(i+1-t)(b+t-i)\al_{t-1},\quad 0<t\leq i.
$$
Setting $\al_0=1$, we get the explicit solution
$$
\al_t=\underset{1\leq j\leq t}\prod \frac{(i+1-j)(b+j-i)}{j(a+1-j)},\quad 1\leq t\leq i.
$$

For instance, let $a=2$ and $b=4$. Then
$$
E(0)=\left(\begin{array}{ccccc}
             0 & 0 & 0 & 0 & 0\\
             0 & 0 & 0 & 0 & 0\\
             1 & 0 & 0 & 0 & 0\\
             \end{array}
         \right), E(1)= \left(\begin{array}{ccccc}
             0 & 0 & 0 & 0 & 0\\
             2 & 0 & 0 & 0 & 0\\
             0 & 1 & 0 & 0 & 0\\
             \end{array}
         \right),E(2)=\left(\begin{array}{ccccc}
             6 & 0 & 0 & 0 & 0\\
             0 & 3 & 0 & 0 & 0\\
             0 & 0 & 1 & 0 & 0\\
             \end{array}
         \right).
$$
Given $\al,\la\in F$, set $A=J^3(\al)\oplus J^5(\al-\la)$ and view $\M_{3,5}$ as an $F[t]$-module via the operator $\ad_{\gl(8)}A-\la 1_{\gl(8)}$.
Then
$$
\M_{3,5}=F[t] \widehat{E(0)}\oplus F[t] \widehat{E(1)}\oplus F[t] \widehat{E(2)},
$$
where $\widehat{E(0)},\widehat{E(1)},\widehat{E(2)}$ have minimal polynomials $t^7,t^5,t^3$ with respect to the action of $\ad_{\gl(8)}A-\la 1_{\gl(8)}$, or $e$, on $\M_{3,5}$.

Suppose $x$ acts on $V$ via $J^7(\la)\oplus J^5(\la)\oplus J^3(\la)$ on $V$, with respective $F[x]$-generators $v,w,u$
Then
$$
x\mapsto A, v\mapsto \widehat{E(0)}, w\mapsto \widehat{E(1)}, u\mapsto \widehat{E(2)}
$$
defines a faithful uniserial representation $R:\g\to\gl(8)$ extending $R_{\al,3,0}$ (which is defined by $x\mapsto A, v\mapsto \widehat{E(0)}$).

Replace $J^5(\la)\oplus J^3(\la)$ in the above example by $J^4(\la)\oplus J^2(\la)$. To obtain a representation of the new $\g$,
simply let $w\mapsto e\cdot \widehat{E(1)}$ and $u\mapsto e\cdot \widehat{E(2)}$.

Replace $J^5(\la)\oplus J^3(\la)$ in the first example by $J^1(\la)$. To obtain a representation of this latest $\g$,
we may let $w\mapsto e^4\cdot \widehat{E(1)}+e^2\cdot \widehat{E(2)}$ (there is no $u$ in this case).


\end{document}